\documentclass[12pt, reqno]{amsart}

\usepackage{amsthm,amsmath,amsthm,amssymb,mathrsfs,graphicx,mathtools,relsize,hyperref,cleveref,enumerate,nameref, geometry}
\usepackage[utf8]{inputenc}
\usepackage[all]{xy}

\theoremstyle{theorem}
 \newtheorem{thm}{Theorem}[section]
 \newtheorem{prop}[thm]{Proposition}
 
 \newtheorem{cor}[thm]{Corollary}

\theoremstyle{definition}
 
 \newtheorem{dfn}[thm]{Definition}
\theoremstyle{remark}
 \newtheorem{rem}[thm]{Remark}
 \numberwithin{equation}{section}
 
\renewcommand{\le}{\leqslant}\renewcommand{\leq}{\leqslant}
\renewcommand{\geq}{\geqslant}

\def\Aut{\text{\rm Aut}}

\def\Prim{\text{\rm Prim}}
\def\id{\text{\rm id}}

\def\supp{\text{\rm supp}}
\def\ker{\text{\rm ker}}
\def\cker{C^*\text{\rm -ker}}

\def\ZZ{\mathbb{Z}}
\def\NN{\mathbb{N}}
\def\C{\mathbb{C}}
\def\R{\mathbb{R}}
\def\Q{\mathbb{Q}}
\def\F{\mathbb{F}}

\def\ind{\text{\rm ind}}
\def\supp{\text{\rm supp}}

\def\I{\text{\rm I}}

\title[Type I two-step nilpotent groups]{The type I dichotomy\\ 
for two-step nilpotent locally compact groups}
\date{June 9, 2026}

\thanks{The authors acknowledge support from the FWO and F.R.S.-FNRS under the Excellence of Science program (project ID 40007542).}

\author[Pierre-Emmanuel Caprace]{Pierre-Emmanuel Caprace} 
\address{Institut de recherche en mathématique et physique \\
Chemin du Cyclotron 2 \\
boîte L7.01.02 \
Université catholique de Louvain \\ 
1348 Louvain-la-Neuve \\
Belgique.}
\email{pierre-emmanuel.caprace@uclouvain.be}

\author[Max Carter]{Max Carter} 
\address{Institut de recherche en mathématique et physique \\
Chemin du Cyclotron 2 \\
boîte L7.01.02 \\
Université catholique de Louvain \\ 
1348 Louvain-la-Neuve \\
Belgique.}
\email{max.carter@uclouvain.be}

\begin{document}

\begin{abstract}
We address the type~$\I$ dichotomy for two-step nilpotent locally compact groups. Invoking work of Baggett--Kleppner, we characterize the closed points of the unitary dual of such a group $G$ purely in terms of the group structure. An algebraic criterion characterizing when $G$ is  a type~$\I$ group is derived. We show that this criterion automatically holds if $G$ is a central extension of vector groups over a non-discrete locally compact field $k$ such that the commutator map is $k$-bilinear. As an application, we show that the unipotent radicals of minimal parabolics in simple algebraic groups of $k$-rank one are type~$\I$ groups. We also discuss the type~$\I$ dichotomy for $p$-torsion contraction groups, and exhibit, for each prime $p$, uncountably many pairwise non-isomorphic such groups that are not type~$\I$. This answers  a recently posed question by the second author. Finally, we adapt a recent construction of Chirvasitu to obtain numerous examples of two-step nilpotent torsion locally compact groups that are not type~$\I$, but that embed as closed cocompact normal subgroups in  two-step nilpotent  groups that are type~$\I$.
\end{abstract}

\maketitle

\begin{flushright}
\begin{minipage}[t]{0.60\linewidth}\itshape\small
The precise boundary between good and bad groups is not well defined, and varies with the amount of technical complication you can tolerate. 
\vspace{2mm}

\hfill\upshape \textemdash Roger E. Howe \cite[p.~307]{How77}, 1977.
\end{minipage}
\end{flushright}

\section{Introduction}

A topological group $G$ is called \textbf{type $\I$} if every unitary representation of $G$ generates a type $\I$ von Neumann algebra \cite[$\S$5]{Dix77}. The classical theory of unitary representations of locally compact groups shows that there is a strong dichotomy between type $\I$ groups  and those which are not type $\I$, that we shall henceforth call the \textbf{type~$\I$ dichotomy}: type $\I$ groups are precisely those groups all of whose unitary representations decompose \textit{uniquely} into a direct integral of irreducible unitary representations \cite[$\S$8]{Dix77}. Kirillov suggested to use the terms \textbf{tame} vs.\ \textbf{wild} instead of type~$\I$ vs.\ non-type~$\I$, see \cite[\S 8.F.b, Comment~(13)]{BH20}.  Also, a result of Glimm, referred to as \textit{Glimm's theorem} (see \cite[Theorem 7.6]{Fol16} or \cite{Gli61}), implies that a second countable locally compact group $G$ is   type $\I$ if and only if its unitary dual, $\widehat{G}$, endowed with Fell's topology, is a $T_0$ topological space. Another natural, and formally stronger condition, is that the unitary dual $\widehat G$ be a $T_1$ space. In that case, the group $G$ is called \textbf{CCR}  (see \cite[$\S$7.2]{Fol16}).

It is thus an important problem to determine which locally compact groups satisfy the type $\I$ property or the stronger CCR property. Classically, it was Lie groups and algebraic groups over local fields where this problem received the most attention, and for many of these groups, the solution to this problem is well understood \cite{Kir62,War72-1,War72-2,Lip75,BE21}. One notable exception, pointed out by Bekka--Echterhoff \cite[Remark~1(iv)]{BE21}, is the class of unipotent algebraic groups over local fields of positive characteristic: it is currently unknown whether all these groups are of type~$\I$. 

In this paper, we address the type $\I$ dichotomy for \textbf{two-step nilpotent groups}, i.e.\ locally compact  groups whose commutator subgroup is central (equivalently, they are nilpotent of nilpotency class at most~$2$). Classical results due to J.~Dixmier \cite{Dix59} and A.~Kirillov~\cite{Kir62} ensure that every connected nilpotent Lie group is CCR, hence of type~$\I$. On the other hand, Thoma's theorem ensures that a discrete group is of type~$\I$ if and only if it is virtually abelian (in particular every discrete group of type~$\I$ is CCR), see \cite{Tho68} or \cite[Theorem~7.D.1]{BH20}. In particular, there exist two-step nilpotent groups that are not of type~$\I$, e.g. the Heisenberg group over~$\ZZ$. 

Since the type~$\I$ and CCR properties are characterized by a separation condition on the unitary dual, it is a natural question to investigate the closed points  in $\widehat{G}$. We denote by $[\pi]$ the class of an irreducible representation with respect to unitary equivalence. 

\begin{thm}[See Theorem~\ref{thm:closed-points}]\label{thmintro:closed-point}
Let $G$ be a  two-step nilpotent second countable locally compact group $G$ with center $Z$, and $\pi$ an irreducible unitary representation of $G$ with central character $\chi$ (see \cite[Definition~1.A.12]{BH20}). Then the following assertions are equivalent:
\begin{enumerate}[(i)]
\item $[\pi]$ is a closed point in the unitary dual $\widehat G$;

\item The homomorphism $\omega_\chi \colon G/Z \to \widehat{G/Z}$, defined by setting $\omega_\chi(gZ)(hZ) = \chi([g, h])$, has a closed image. 
\end{enumerate}
\end{thm}

The proof relies on a reformulation of results due to Baggett--Kleppner~\cite{BK73} concerning projective representations of abelian groups. We also rely  on the so-called \emph{Poguntke parametrization} of the primitive dual of $G$, that will be recalled in Theorem~\ref{thm:Pog} below. The following characterization of the type~$\I$ property purely in terms of the algebraic/topological structure of $G$ follows readily.

\begin{cor}[See Corollary~\ref{cor:type-I-char}]\label{corintro:type-I-char}
For  a  two-step nilpotent second countable locally compact group $G$ with center $Z$, the following conditions are equivalent: 
\begin{enumerate}[(i)]
\item $G$ is type~$\I$;
\item $G$ is CCR;
\item For each $\chi \in \widehat Z$, the homomorphism $\omega_\chi \colon G/Z \to \widehat{G/Z}$ has closed image. 
\end{enumerate}
\end{cor}

The equivalence between (i) and (ii) follows directly from a general fact, due to Poguntke, that the primitive dual of a two-step nilpotent locally compact group is a $T_1$ space (see Corollary~\ref{cor:T1} below). 

Several variants of the closedness condition in Theorem~\ref{thmintro:closed-point}(ii) will be established for totally disconnected groups, see Proposition~\ref{prop:tdlc-step2-nilp}. This yields the following consequence for discrete countable groups that is of independent interest. 

\begin{cor}[See Corollary~\ref{cor:G-discrete}]
For  a  two-step nilpotent  countable  group $G$ with center $Z$, the following conditions are equivalent:
\begin{enumerate}[(i)]
\item $G$ is type~$\I$;
\item $G$ is virtually abelian; 
\item For each character $\chi \in \widehat Z$, the quotient $G/\ker(\chi)$ is center-by-finite. 
\end{enumerate}
\end{cor}
	
The equivalence between (i) and (ii) is valid for all discrete groups by virtue of Thoma's theorem mentioned above. The equivalence between (ii) and (iii) is a purely algebraic statement that will be useful in Section~\ref{sec:contraction}. 

Corollary~\ref{corintro:type-I-char} allows us to recover, in a unified way, results on connected  groups and on $p$-adic analytic groups, see Corollary~\ref{cor:connected+p-adic}. However, its main interest is that it also covers $p$-torsion groups, where $p$ is an arbitrary prime. This will be concretized through the following statement. 

\begin{thm}[See Theorem~\ref{thm:bilinear}]\label{thmintro:bilinear}
Let $k$ be a non-discrete locally compact field, and $A, N$ be finite-dimensional $k$-vector spaces. Let $G$ be a locally compact group that is a central extension of $A$ by $N$. If the map $A \times A \to N$ induced by the commutator map on $G$ is $k$-bilinear, then $G$ is type~$\I$. 
\end{thm}

It is important to keep in mind that the commutator map $A \times A \to N$ is automatically bi-additive. Theorem~\ref{thmintro:bilinear} allows us to establish, in a uniform fashion, that several natural classes of unipotent algebraic groups over local fields of arbitrary characteristic are type~$\I$.

\begin{cor}\label{corintro:algebraic}
Let $k$ be a non-discrete locally compact field. 
The following locally compact two-step nilpotent groups are all type~$\I$:
\begin{enumerate}[(1)]
\item The $2n+1$-dimensional Heisenberg group over $k$, for all $n \geq 1$. 
\item Given a two-step nilpotent Lie algebra $\mathfrak g$ over $k$ with Lie bracket $[\cdot, \cdot]_{\mathfrak g}$, the unipotent group $E(\mathfrak g)$ with underlying set $\mathfrak g$ and multiplication defined by $v.w = v+ w +  [v, w]_{\mathfrak g}$.
\item The unipotent radical of a minimal $k$-parabolic subgroup in an absolutely simple algebraic $k$-group of $k$-rank one. 
\end{enumerate}
\end{cor}

\begin{rem} \ 
\begin{enumerate}[(1)]
\item In view of Corollary~\ref{corintro:type-I-char}, the item~(1) in Corollary~\ref{corintro:algebraic} recovers and strengthens \cite[Theorem~5.15]{Car24}. 

\item The map $\mathfrak g \mapsto E(\mathfrak g)$ in Corollary~\ref{corintro:algebraic}(2) is closely related to the so-called \textbf{Lazard correspondence}. The group  $E(\mathfrak g)$ is abelian if  $\mathrm{char}(k) = 2$, but non-abelian if $\mathrm{char}(k) \neq 2$ and  $\mathfrak g$ is non-commutative. 

\item The groups in Corollary~\ref{corintro:algebraic}(3) include non-abelian groups in all characteristics.

\end{enumerate}

\end{rem}

\begin{rem}
As mentioned above, it is currently unknown whether all unipotent algebraic groups over local fields of positive characteristic are type~$\I$, see  \cite[Remark~1(iv)]{BE21}. An alternative approach to this problem is provided by the work of Echterhoff--Kl{\"u}ver~\cite{EK12}, developing the Kirillov orbit method for rather general nilpotent groups, following pioneering work of R.~Howe \cite{How77}. This approach has the advantage that it applies to nilpotent groups of higher nilpotency class $c$, and that it provides a description of the unitary dual for type~$\I$ groups. However, there are  several important technical restrictions. One of  them is that the group should not contain $p$-torsion for any prime $p \leq c$; another one is a condition of \textit{regularity} that is rather restrictive in positive characteristic\footnote{In Example~8.2 of \cite{EK12}, it is claimed that if $p$ is a prime and $G$ is a nilpotent locally compact group of nilpotency class $c< p$ in which every element is contained in a pro-$p$ subgroup, then the regularity condition automatically holds for $G$. That claim is  inaccurate: indeed, in the additive group $\mathfrak g= \mathbb F_p(\!(t)\!)$, the sum of two closed subgroups can fail to be closed (e.g. the  subgroups  $\mathfrak a = \mathbb F_p[t^{-1}]$ and $\mathfrak b = \mathbb F_p[t^{-1} + t]$ are both discrete, and their sum is dense).}. 
\end{rem}

The last part of this paper aims at providing additional explicit examples of two-step nilpotent groups illustrating the type~$\I$ dichotomy. We focus on groups that are central extensions of the additive group $A = \F_p(\!(t)\!)$ by itself, where $p$ is an arbitrary prime. We require that the commutation relations induce a map $A \times A \to A$ that sends pairs $(t^a, t^b)$   to monomials (see Definition~\ref{def:monomial}). We establish criteria for such groups to be type~$\I$ (see Theorem~\ref{thm:monomial-type-I}) and non-type~$\I$ (see Theorem~\ref{thm:non-type-I}). As an application, we obtain the following result concerning the family of two-step nilpotent groups $A \times_{\eta_s } A$ introduced by Gl\"ockner--Willis \cite[Section~8]{GW21}. A group in that family is defined by a cocycle $\eta_s$ determined by a sequence $s \colon \NN_{>0} \to \{0,1\}$. The precise definition will be recalled in Section~\ref{sec:GW}.

\begin{thm}\label{thmintro:GW-type-I}
Let  $p$ be prime. Let $s\in \{0,1\}^{\mathbb{N}_{>0}}$ and  $G = A \times_{\eta_s} A$ the group mentioned above. 

Then $G$ is type~$\I$ as soon as the sequence $s$ satisfies any of the following conditions  for some integer $c \geq 0$.
\begin{enumerate}[(1)]

\item  $s(z) =1 $ for all $z$ with $z > c$. 

\item $s(z)= 0$ for all even $z$, and $s(z) =1 $ for all  odd $z$ with $z > c$. 

\item There is an integer $d>0$ such that $s(z) = 0$ for all $z \not \in d\NN$, and  $s(z) =1$ for all   $z \in d\NN$ with $z > c$. \end{enumerate}

On the other hand,  $G$ is not type~$\I$ if $s$ satisfies the following.
\begin{enumerate}[(1)] \setcounter{enumi}{3} 
\item For some integer $d>0$, we have  $s(d)=1$ and $s(dn) = 0$  for all $n \geq 2 $.

\end{enumerate}

\end{thm}

It is proved in \cite[Theorem~5.9]{Car24} that if  $s$ is non-zero and finitely supported, then $G$ is not type~$\I$. Theorem~\ref{thmintro:GW-type-I}(4) recovers and strengthens that fact. 

The following consequence of Theorem~\ref{thmintro:GW-type-I} (see also  Proposition~\ref{prop:GW}) provides a positive solution to \cite[Problem~5.13]{Car24}. 

\begin{cor}
For a given prime $p$, the family of groups $A \times_{\eta_s} A$ contains infinitely many pairwise non-isomorphic type~$\I$ groups, and uncountably many pairwise non-isomorphic non-type~$\I$ groups. 
\end{cor}

\begin{rem}
The sequences giving rise to type~$\I$ groups in Theorem~\ref{thmintro:GW-type-I} are asymptotically periodic, but not asymptotically zero. However,  the precise boundary between type~$\I$ groups and non-type~$\I$ groups remains unclear. Notice in particular that, according to
Theorem~\ref{thmintro:GW-type-I}, the group $G$ is type~$\I$ if $s = (1, 0, 1, 0, 1, 0, 1, \dots)$ but not type~$\I$ if $s= (1, 1, 1, 0, 1, 0, 1, \dots)$. 
\end{rem}

The class of type~$\I$ groups enjoys various stability properties with respect to basic algebraic constructions: e.g. passing to open subgroups, passing to quotient groups, forming finite direct products,  see \cite[Proposition~6.E.21]{BH20}. It is moreover known that if $G$ is a locally compact group admitting a closed normal subgroup $N$ such that $G/N$ is compact, then $G$ is type~$\I$ as soon as $N$ is so. Recently, A.~Chirvasitu \cite[\S 2.2]{Chi23} has shown that the converse does not hold, even within the class of two-step nilpotent groups: in general the type~$\I$ property need not be inherited by a closed cocompact normal subgroup. By revisiting Chirvasitu's construction in the light of the results of this paper, we establish the following. It applies notably to all the contraction groups appearing in Theorem~\ref{thmintro:GW-type-I}, that are further discussed in Section~\ref{sec:GW}.

\begin{thm}\label{thmintro:type-I-ext}
Let $G$ be a two-step nilpotent second countable locally compact group with center $Z$. Let $N \leq G$ be a closed subgroup with $[G, G] \leq N \leq Z$. Let $p$ be a prime, and suppose that $N$ and $G/N$ both have exponent~$p$. 

Then for each character~$\chi \in \widehat N$, the quotient $G/\ker(\chi)$ continuously embeds as a closed cocompact normal subgroup in a  two-step nilpotent locally compact group that is type~$\I$. 
\end{thm}

For all non-type~$\I$ groups arising in Theorem~\ref{thmintro:GW-type-I}, there is a central character $\chi$ such that $G/\ker(\chi)$ is not type~$\I$ (see Remark~\ref{rem:ex-ext}). Hence Theorem~\ref{thmintro:type-I-ext} yields a broad family of non-type~$\I$ groups embedding as closed cocompact normal subgroups in  two-step nilpotent locally compact groups. 

\medskip		
While we have focused on second countable groups in this introduction, several results mentioned above are established below without any countability assumptions. 

\subsection*{Layout of the article}

In Section \ref{sec:nilp} we establish  Theorem~\ref{thmintro:closed-point} and its consequences, and present various supplements for the class of totally disconnected groups.  In Section \ref{sec:contraction}, we focus on two-step nilpotent groups with monomial commutation relations, finishing with a proof of Theorem~\ref{thmintro:GW-type-I}. Finally, the proof of Theorem~\ref{thmintro:type-I-ext} is given in Section~\ref{sec:ext}. 

\subsection*{Acknowledgements}

We thank Tom De Medts for drawing our attention to his preprint~\cite{TDM}, and for his comments on the proof of Corollary~\ref{corintro:algebraic}. We are grateful to the referee for their comments on an earlier version of this article.


\section{Unitary representation theory of  two-step nilpotent groups}\label{sec:nilp}

Throughout this paper,  we define the commutator of two elements $g, h$ of a group as $[g, h] = g h g^{-1} h^{-1}$. We assume that the reader has some familiarity with the theory of unitary representations of locally compact groups (see \cite{Dix77} and \cite{BH20}) and with Pontryagin duality (see \cite[Ch.~II]{Bou67} and \cite[Ch.~F]{Str06}). 

\subsection{Characterizing closed points in the unitary dual}

Given an irreducible unitary representation $\pi$ of a locally compact group $G$, we denote by $[\pi]$ its equivalence class under unitary equivalence. Thus $[\pi] \in \widehat G$. In the special case where $\pi$ is one-dimensional, we identify $[\pi]$ with the character of $\pi$, which is a complex valued function on $G$. For each closed normal subgroup $L$ of $G$, we define $L^\perp \subseteq \widehat G$ as the closed subset consisting of those characters of $G$ that are trivial on $L$. Note that, if $[G,G] \le L$, then $L^\perp$ is naturally isomorphic to $\widehat{G/L}$, and carries a canonical group structure defined by pointwise multiplication of characters. 
By Schur's lemma, the restriction of an irreducible unitary representation $\pi$ to the center $Z$ of $G$ defines an action of $Z$ by scalar operators. Thus $\pi|_{Z}(z) = \chi_\pi(z) \mathrm{Id}$ for some character $\chi_\pi \in \widehat Z$, called the \textbf{central character} of $\pi$. 

We start by recalling Poguntke's parametrization of the primitive ideal space $\Prim(G)$ of a two-step nilpotent group $G$, which appears as the first proposition in \cite[Part I]{Pog83} and, according to Poguntke,  relies  on ideas due to R.~Howe, see \cite[Prop.~5]{How77} and Kaniuth \cite[Lemma~2]{Kan82}.  We follow the presentation from \cite[\S 5.8]{KT12}. 

So let $G$ be a two-step nilpotent locally compact group with center $Z$. Every character $\chi \in  \widehat Z$ defines a homomorphism 
\begin{equation}\label{eqn:wchi} \omega_\chi \colon G \to Z^\perp : g \mapsto \big( xZ \mapsto \chi([g, x]) \big). \end{equation}
Since $\omega_\chi$ is trivial on $Z$ and since  $Z^\perp$ is canonically isomorphic to $\widehat{G/Z}$, we may as well view $\omega_\chi$ as a homomorphism of abelian groups
$$\omega_\chi \colon G/Z \to \widehat{G/Z}.$$ 

Given $\chi \in \widehat Z$, we set $L_\chi = \{g \in G\mid [g,  G] \subseteq \ker(\chi)\}$. Thus $L_\chi$ is the preimage under the canonical projection $G \to G/\ker(\chi)$ of the center of the quotient group $G/\ker(\chi)$. Clearly we have $\ker(\chi) \leq Z \leq L_\chi$. We also set
$$A_\chi =\{ \alpha \in \widehat{L_\chi} \mid \alpha|_Z = \chi\}.$$
Since $L_\chi/\ker(\chi)$ is abelian, and since characters of closed subgroups of locally compact abelian groups can always be extended (see \cite[Chap.~II, \S 1, no. 7, th.~4]{Bou67}), we infer that $A_\chi$ is non-empty. Set 
$$\mathcal P = \{(\chi, \alpha) \mid \chi \in \widehat Z, \ \alpha \in A_\chi\}.$$
The first part of the following result is known as \textbf{Poguntke's parametrization} of the primitive ideal space $\Prim(G):=\Prim(C^*(G))$. 

\begin{thm}\label{thm:Pog}
Let $G$ be a two-step nilpotent locally compact group. 
\begin{enumerate}[(i)]
\item The map
$$(\chi, \alpha)  \mapsto \cker(\ind_{L_\chi}^G \alpha)$$
establishes a one-to-one correspondence from $\mathcal P$ to $\Prim(G)$. 

\item For all $(\chi_1, \alpha_1), (\chi_2, \alpha_2) \in \mathcal P$, the weak containment $\ind_{L_{\chi_1}}^G \alpha_1 \preceq \ind_{L_{\chi_2}}^G \alpha_2$ is equivalent to the equality $(\chi_1, \alpha_1) =(\chi_2, \alpha_2)$. 

\item Let $\pi$ be an irreducible unitary representation of $G$ with central character $\chi$. Then 
$\pi $ is weakly equivalent to $ \pi \otimes \psi$ for all $\psi \in L_\chi^\perp$. Moreover, the group $L_\chi^\perp$ coincides with the closure of the image of $\omega_\chi$. 

\end{enumerate}
	
\end{thm}

\begin{proof}
For (i), see \cite[Theorem~5.66]{KT12}. The item (ii) follows from the arguments in the last part of the proof of \cite[Theorem~5.66]{KT12}. The item (iii) is explicitly established in the same proof (see the last few lines on page 260 and the first few lines on page 261 in \cite{KT12}). 
\end{proof}

\begin{cor}\label{cor:T1}
Let $G$ be a two-step nilpotent locally compact group. Given irreducible unitary representations $\pi_1, \pi_2$ of $G$, we have $\pi_1 \preceq \pi_2$ if and only if $\pi_1 \sim \pi_2$. Equivalently,  the space $\Prim(G)$ is $T_1$. 

In particular, a two-step nilpotent locally compact group is type~$\I$ if and only if it is CCR. 
\end{cor}
\begin{proof}
This follows directly from Theorem~\ref{thm:Pog}(i) and (ii). 
\end{proof}

\begin{rem}
The  fact that $\Prim(G)$ is a $T_1$-space is pointed out by D.~Poguntke in the first paragraph of Part I in \cite{Pog83}. This property is also known for discrete nilpotent groups (see \cite{How77} for finitely generated groups and \cite{Pog82} for the general case) and for nilpotent locally compact groups containing an open normal subgroup that is compactly generated (see \cite{CaMo86,Lud86}). It is an intriguing open problem to determine whether it holds for an arbitrary nilpotent locally compact group. 
\end{rem}
			
Recall that a unitary representation $\rho$ of $G$ is called a \textbf{factor representation} if the von Neumann algebra generated by $\rho(G)$ is a factor. In that case, the operator $\rho(z)$ is scalar for any element $z$ of the center $Z$ of $G$. Thus $\rho|_Z(z) = \chi_\rho(z) \mathrm{Id}$ for some character $\chi_\rho \in \widehat Z$, called the \textbf{central character} of $\rho$. This definition is consistent with the one given above: indeed every irreducible representation is a factor representation. Combining Theorem~\ref{thm:Pog} with the work by Baggett--Kleppner~\cite{BK73} (valid without any second countability assumption), we  establish the following characterization of the closed points in $\widehat G$. 
 
\begin{thm}\label{thm:closed-points}
Let $G$ be a locally compact group that is two-step nilpotent, and $\pi$ an irreducible unitary representation of $G$ with central character $\chi$. Then the following assertions satisfy the implications (i)~$\Leftarrow$~(ii)~$\Leftrightarrow$~(iii)~$\Leftrightarrow$~(iv). If in addition $G$ is second countable, then they are all equivalent, and the second sentence of (ii) can be discarded.  
\begin{enumerate}[(i)]
\item $[\pi]$ is a closed point in the unitary dual $\widehat G$. 

\item The homomorphism $\omega_\chi \colon G/Z \to \widehat{G/Z}$, defined as in Equation~\ref{eqn:wchi}, has a closed image. Moreover,  the image of every open subset is relatively open.

\item Every factor representation of $G$ with central character $\chi$ is type~$\I$.

\item Every factor representation of $G$  weakly equivalent to $\pi$ is type~$\I$.
\end{enumerate}
	
 \end{thm}
 
 \begin{proof}
 We may view $G$ as a central extension of the abelian group $G/Z$ associated with a $Z$-valued $2$-cocycle $\kappa \colon G/Z \times G/Z \to Z$. The composite map $\omega = \chi \circ \kappa$ defines a $\mathbf C$-valued cocycle, and there is a canonical one-to-one correspondence between the unitary representations of $G$ with central character~$\chi$, and the so-called $\omega$-representations of $G/Z$ considered by Baggett--Kleppner~\cite{BK73} (see the first few lines on p.~302 in \cite{BK73}). 
 
The image of $\omega_\chi$ is contained in $L_\chi^\perp$, which is a closed subset of $Z^\perp$. By the Corollary to Theorem~3.2 in \cite{BK73}, we infer that (ii) and (iii) are equivalent (see also the discussion preceding Theorem~3.3 in \cite{BK73}, showing that the second sentence of (ii) can be discarded if $G$ is second countable). 

Since the operation of restricting representations preserves weak containment, every factor representation weakly contained in $\pi$ must have the same central character as $\pi$. Thus it is clear that (iii) implies (iv). 

Let us now show that  (iii) implies (i). Let $\pi$ be an irreducible unitary representation of $G$. Since $\Prim(G)$ is $T_1$ by Corollary~\ref{cor:T1}, every primitive ideal in the $C^*$-algebra of $G$ is maximal, hence the $C^*$-algebra generated by $\pi$, denoted by $C^*(\pi)$, is simple. Every factor representation of $C^*(\pi)$ naturally defines a factor representation of $G$ that is weakly contained in $\pi$. Considering the restriction of such a factor representation to $Z$, we infer that its central character is $\chi$ since the operation of restriction preserves weak containment. If (iii) holds this factor representation is type~$\I$. From Glimm's theorem and its extension to the non-separable case by Sakai (see \cite[Theorem 9.1]{Dix77} and the discussion in \cite[IV.1.5.8]{Bla06}), we know that a $C^*$-algebra is type~$\I$ if and only if each of its factor representations is type~$\I$. We deduce that the  $C^*$-algebra $C^*(\pi)$ is type~$\I$. Since $C^*(\pi)$ is also simple, it must be \textit{elementary}, i.e. it is isomorphic to the $C^*$-algebra of compact operators on a Hilbert space. It follows that $C^*(\pi)$ has a unique equivalence class of non-zero irreducible representations (see \cite[Corollary~4.1.5]{Dix77}). In other words, every irreducible unitary representation of $G$ weakly contained in $\pi$ is equivalent to $\pi$. This proves that (iii) implies (i). 

Let us now show that (iv) implies (iii).  We have already seen that (ii) and (iii) are equivalent. Assume that they fail. By Theorem~\ref{thm:Pog}(i), there exists  $\alpha \in A_\chi$  such that $\pi \sim \ind_{L_\chi}^G  \alpha $ (since $\pi(\ell)$ is a scalar operator for each $\ell \in L_\chi$, this just means that $\alpha \in \widehat{L_\chi}$ is the associated character). We have $\ker(\pi ) = \ker(\alpha) \leq L_\chi$, hence $\ker(\alpha)$ is normal in $G$ and the representation $\pi$ factors through $G/\ker(\alpha)$. The center of $G/\ker(\alpha)$ is $L_\chi/\ker(\alpha)$. Recall that $\omega_\chi$ may be viewed as a homomorphism of $G/L_\chi$ to its dual. Using the equivalence between  (ii) and (iii) for the quotient group $G/\ker(\alpha)$, we deduce  that some factor representation $\rho$ of $G/\ker(\alpha)$ with central character $\alpha$ is not type~$\I$. By precomposing $\rho$ with the canonical projection $G \to G/\ker(\alpha)$, we may view $\rho$ as a factor representation of $G$. By definition, for each $\ell \in L_\chi$ we have $\rho(\ell) = \alpha(\ell) \mathrm{Id}$. This implies that each irreducible representation $\sigma \preceq \rho$ has the same Ponguntke parameters as $\pi$. Therefore, we have  $\sigma \sim \pi$ by Theorem~\ref{thm:Pog}. 
It follows that the factor representation $\rho$ is weakly equivalent to $\pi$. Since $\rho$  is not type~$\I$, we deduce that  (iv) indeed fails. 

It remains to show that (i) implies (ii) under the additional hypothesis that $G$ is second countable. If (iv) fails, then $G$ has a  factor representation $\rho$ weakly equivalent to $\pi$, that is not type~$\I$. 
Invoking \cite[Corollary~7.F.4]{BH20}, it follows that $G$ has uncountably many inequivalent irreducible representations weakly equivalent to $\rho$. Each of them is weakly equivalent to $\pi$, so that   the closure $\overline{\big\{[\pi]\big\}}$ contains uncountably many points. Thus (i) fails. 
 \end{proof}
	
\begin{rem}
In order to prove the implication (i)~$\Rightarrow$~(ii) in Theorem~\ref{thm:closed-points} without the second countability assumption, we would need to know that if the simple $C^*$-algebra $C^*(\pi)$ has a unique equivalence class of irreducible representations, then it is type~$\I$ (hence elementary). This is not true for non-separable simple $C^*$-algebras in general (see the counterexample to Naimark's problem in \cite{AW04}). We do not know whether this property holds for the family of simple $C^*$-algebras $C^*(\pi)$ arising  in Theorem~\ref{thm:closed-points}. 
\end{rem}
 
The following remark will be used frequently in the sequel. 

\begin{rem}\label{rem:N}
Let  $G$ be a   two-step nilpotent  locally compact group with center $Z$, and let $N$ be a closed normal subgroup with $[G, G] \leq N \leq Z$. Each character $\chi \in \widehat Z$ may be viewed as a character of $N$ by restriction. Moreover for each $\chi \in \widehat Z$, the homomorphism $\omega_\chi$ is trivial on $Z$-cosets, hence also on $N$-cosets. Thus we may naturally view $\omega_\chi$ as a map from $G/N \to \widehat{G/N}$. Conversely, given a character $\psi \in \widehat N$, the assignment (\ref{eqn:wchi}) defines a homomorphism $\omega_\psi \colon G/N \to \widehat{G/N}$. We can extend $\psi$ continuously to a character $\chi \in \widehat Z$ (see \cite[Chap.~II, \S 1, no. 7, th.~4]{Bou67}), and obtain thereby a homomorphism $\omega_\chi \colon G/Z \to \widehat{G/Z}$. Since  $\widehat{G/Z} \cong Z^\perp \subseteq N^\perp \cong \widehat{G/N}$, we may view $\widehat{G/Z}$ as a closed subspace of $\widehat{G/N}$. In this way, we see that the image of $\omega_\psi$ is closed if and only if the image of $\omega_\chi$ is closed. 
\end{rem}

The following characterization of two-step nilpotent groups of type~$\I$ follows easily. 

\begin{cor}\label{cor:type-I-char}
Let  $G$ be a   two-step nilpotent  locally compact group with center $Z$, and let $N$ be a closed normal subgroup with $[G, G] \leq N \leq Z$. The following conditions are equivalent. 
\begin{enumerate}[(i)]
\item $G$ is type~$\I$.
\item $G$ is CCR. 
\item For each $\chi \in \widehat Z$, the homomorphism $\omega_\chi \colon G/Z \to \widehat{G/Z}$ has closed image. Moreover,  the image of every open subset is relatively open.
\item For each $\psi \in \widehat N$, the homomorphism $\omega_\psi \colon G/N \to \widehat{G/N}$ has closed image. Moreover,  the image of every open subset is relatively open.
\end{enumerate}
If in addition $G$ is second countable, then this is also equivalent to:
\begin{enumerate}[(i)] \setcounter{enumi}{4} 
\item For each $\chi \in \widehat Z$, the homomorphism $\omega_\chi \colon G/Z \to \widehat{G/Z}$ has closed image. 

\item For each $\psi \in \widehat N$, the homomorphism $\omega_\psi\colon G/N \to \widehat{G/N}$ has closed image. 
\end{enumerate}
\end{cor}
\begin{proof}
The equivalence between (i) and (ii) was recorded in Corollary~\ref{cor:T1}. The equivalence between (i), (iii) and (v) follows from the equivalence between (ii) and (iii) in Theorem~\ref{thm:closed-points} (since a group is type~$\I$ if and only if each of its factor representations is type~$\I$, see \cite[Theorem~9.1]{Dix77} and the discussion in \cite[IV.1.5.8]{Bla06}). The equivalence with (iv) and (vi) follows, in view of Remark~\ref{rem:N}.
\end{proof}

Corollary~\ref{cor:type-I-char} provides a characterization of the type~$\I$ condition in terms of the topological/algebraic structure of $G$. Various applications will be presented in Section~\ref{sec:app-type-I} below. Let us already record the following. 

\begin{cor}\label{cor:type-I-quot}
Let $G$ be a two-step nilpotent  locally compact group, and $N$ be a closed subgroup contained in the center of $G$. Assume that $N$ splits as the direct product of two closed subgroups $N_1$ and $N_2$ such that $ [G, G]  \subseteq N_2$. Then $G$ is type~$\I$ if and only if $G/N_1$ is type~$\I$. 
\end{cor}
\begin{proof}
By definition, the type~$\I$ property passes to quotient groups, so $G/N_1$ is type~$\I$ as soon as $G$ is. 

Let $Z$ be the center of $G$ and observe that $[G, G] \leq N \leq Z$ and that  $N \cong N_1 \times N_2$  by hypothesis.  Thus $\widehat N \cong \widehat{N_1} \times \widehat{N_2}$. Then, for each $\chi \in \widehat Z$, there exists $\psi \in \widehat{N}$ such that $\psi|_{N_1}$ is trivial and $\psi|_{N_2} = \chi|_{N_2}$. By \cite[Chap.~II, \S 1, no. 7, th.~4]{Bou67}, there exists a character $\chi' \in \widehat{Z}$ extending $\psi$. It then follows from the definition of $\chi'$ that $\omega_{\chi}$ and $\omega_{\chi'}$ have the same image in $\widehat{G/N}$. Assuming that $G/N_1$ is type~$\I$, Corollary~\ref{cor:type-I-char} implies that the image of $\omega_{\chi'}$ in $\widehat{(G/N_1)/(N/N_1)} \cong \widehat{G/N}$ is closed. The result follows. 
\end{proof}
		
\subsection{Totally disconnected groups}\label{sec:tdlc}

We shall now focus on totally disconnected locally compact groups (whose name will henceforth be abbreviated by \textbf{tdlc groups}). In this case, additional algebraic characterizations may be added to Theorem~\ref{thm:closed-points} and Corollary~\ref{cor:type-I-char}. 

Recall that every continuous homomorphism of a tdlc group to a Lie group has an open kernel, since Lie groups have no small subgroups. In particular, every character of a tdlc group has an open kernel. It follows that if $G$ is a  two-step nilpotent tdlc group with center $Z$ and $\chi \in \widehat Z$, then there exists a compact open subgroup $U \leq G$ with $U \cap Z \leq \ker(\chi)$. We infer that the derived group $[U, U]$ is contained in $\ker(\chi)$. 

\begin{prop}\label{prop:tdlc-step2-nilp}
Let $G$ be a second countable two-step nilpotent tdlc group with center $Z$, and let $\chi \in \widehat Z$. Let  $\omega_\chi \colon G/Z \to \widehat{G/Z}$ be the homomorphism defined in Equation~\ref{eqn:wchi}. Let also $U \leq G$ be a compact open subgroup such that $[U, U] \leq \ker(\chi)$. Define 
$$O_{\chi, U} = \big\{g \in G \mid [g, U] \subseteq \ker(\chi)\big\} $$
and 
$$L_{\chi, U} = \big\{g \in O_{\chi, U} \mid [g, O_{\chi, U}]  \subseteq \ker(\chi)\big\}.$$
Let $\tilde\chi$  be a character of $L_{\chi, U}$ that extends $\chi$  (the quotient $L_{\chi, U}/\ker(\chi)$ is an abelian group that contains $Z/\ker(\chi)$ as a closed subgroup, hence such a character $\tilde\chi$ always exists by \cite[Chap.~II, \S 1, no. 7, th.~4]{Bou67}). 

The following assertions are equivalent. 
\begin{enumerate}[(i)]
\item The image of $\omega_\chi \colon G/Z \to \widehat{G/Z}$ is closed. 

\item The quotient $O_{\chi, U}/L_{\chi, U}$ is finite. 

\item The quotient $O_{\chi, U}/\ker(\chi)$ is \textbf{center-by-finite}, i.e.\ its center is of finite index. 

\item The quotient $O_{\chi, U}/\ker(\chi)$ is a type~$\I$ group. 

\item The quotient $O_{\chi, U}/\ker(\tilde\chi)$ is a type~$\I$ group. 

\item The quotient $O_{\chi, U}/\ker(\tilde\chi)$ is center-by-finite.
\end{enumerate}

\end{prop}

\begin{proof}
By definition, the quotient group $L_{\chi, U}/ \ker(\chi)$ is the center of  $O_{\chi, U}/\ker(\chi)$. The equivalence between (ii) and (iii) readily follows. 

Recall that a subgroup of a locally compact group is closed if and only if it is locally closed (see \cite[Ch.~III, \S~2., Prop.~4]{BouTG}). In the dual $\widehat{G/Z}$, an identity neighbourhood is provided by all the characters that vanish on $UZ/Z$. Given $g \in G$, the character $\omega_\chi(gZ)$ vanishes on $UZ/Z$ if and only if $[g, U] \leq \ker(\chi)$, which means  that $g  \in  O_{\chi, U}$. This proves that   $\omega_\chi$ has a closed image if and only if the restriction of $\omega_\chi$ to $O_{\chi, U}$ has a closed image. 

Now we apply Theorem~\ref{thm:Pog}(iii) to the group $O_{\chi, U}/\ker(\chi)$ and the character $\tilde\chi$ defined on its center. By definition, we have 
$$L_{\tilde\chi} = \{g \in O_{\chi, U}/\ker(\chi) \mid [g,  O_{\chi, U}] \subseteq \ker(\tilde\chi)\}.$$
Since $[G, G] \leq Z$, we infer that $L_{\tilde\chi} = L_{\chi, U}/\ker(\chi)$ and that $\omega_{\tilde\chi}$ is nothing but the restriction of $\omega_\chi$ to $O_{\chi, U}$ (which is indeed trivial on $\ker(\chi)$). The map $\omega_{\tilde\chi}$  may be viewed as a map
$$\omega_{\tilde\chi} \colon O_{\chi, U}/ L_{\chi, U} \to \widehat{O_{\chi, U}/ L_{\chi, U}},$$
whose image is dense by Theorem~\ref{thm:Pog}(iii). 

Since $L_{\chi, U} $ contains $U$,  it is an open subgroup of $O_{\chi, U}$. The image of $\omega_\chi $ is thus a countable dense subgroup of the compact group $\widehat{O_{\chi, U}/ L_{\chi, U}}$. We infer that $\omega_{\tilde\chi}$ has a closed image if and only if the countable discrete group $O_{\chi, U}/ L_{\chi, U} $ is compact. The equivalence between (i) and (ii) follows. 

It is clear that (iii) implies (iv). Since the type~$\I$ conditions passes to quotient groups, and since $\ker(\tilde\chi) $ contains $\ker(\chi)$, the assertion (iv) implies (v). 

Finally, we observe that  $L_{\chi, U}/\ker(\tilde\chi)$ is the center of $O_{\chi, U}/\ker(\tilde\chi)$. Hence (iii) and (vi) are equivalent. If (v) holds, then we invoke Corollary~\ref{cor:type-I-char} for the group $O_{\chi, U}/\ker(\tilde\chi)$ and the central character $\tilde\chi$ (viewed as a character defined on the quotient $L_{\chi, U}/\ker(\tilde\chi)$). We infer that the image of the homomorphism $\omega_{\tilde\chi}$ is closed. As seen above, this implies  that $O_{\chi, U}/ L_{\chi, U}$ is finite. Thus  (v) implies (ii).  
\end{proof}

\begin{rem}\label{rem:disquo}
The subgroup $O_{\chi, U}$ is open in $G$ since it contains $U$, and $O_{\chi, U}/\ker(\chi)$ has an open center, namely $L_{\chi, U}/\ker(\chi)$. Hence the kernel $\ker(\tilde\chi)$ is also open and so the quotient  $O_{\chi, U}/\ker(\tilde\chi)$ is always discrete.
\end{rem}
		
\begin{rem}
Proposition~\ref{prop:tdlc-step2-nilp} may be viewed as a manifestation of Mackey's little group method: the unitary representation theory of the larger group $G$ is controlled by that of the subgroups of the form $O_{\chi, U}$, which play the role of the little groups. 

One may actually describe much more precisely how Mackey's little group method applies to the situation treated by Proposition~\ref{prop:tdlc-step2-nilp}. Indeed, for each such $\chi \in \widehat{Z}$ and corresponding $U$, the group $UZ/\ker(\chi)$ is an abelian closed normal subgroup of  $G/\ker(\chi)$, and one may verify that it is regularly embedded. Moreover, given a character $\psi$ of $UZ/\ker(\chi)$ extending $\chi$ on $Z/\ker(\chi)$, it is easy to see that its stabilizer $G_\psi$ (for the natural $G$-action on the dual of $UZ/\ker(\chi)$) coincides with  $O_{\chi, U}$. In view of Theorem~\ref{thm:closed-points}, one can establish Proposition~\ref{prop:tdlc-step2-nilp} by invoking Mackey's theorem \cite[Theorem~3.11]{Mac76}. 
\end{rem}
	
\begin{rem}
The items (ii) and (iii) in Proposition~\ref{prop:tdlc-step2-nilp} depend on the choice of a compact open subgroup $U$ with $[U, U]\leq \ker(\chi)$, while the item (i) does not. Hence the validity (ii) and (iii) for one specific subgroup $U$ implies it for all.  

Fix a descending chain $U = U_0 \geq U_1 \geq \dots$ of compact open subgroup that form a basis of identity neighbourhoods in $G$. For each $n \geq 0$, set $O_n = O_{\chi, U_n}$. The groups $(O_n)$ form an ascending chain of open subgroups. Given $g \in G$, the character $\omega_\chi(g)$ is continuous, hence it has an open kernel. It follows that $\ker(\omega_\chi(g))$ contains $U_n$ for some $n$. Hence $g \in O_n$. Therefore we have $G = \bigcup_n O_n$. We have just seen that $O_0/\ker(\chi)$ is center-by-finite if and only if $O_n/\ker(\chi)$ is center-by-finite for all $n$. Thus (iii) implies that $G/\ker(\chi)$ is  a countable ascending union of open subgroups that are center-by-finite. 
\end{rem}
	
\begin{rem}\label{rmk:G-discrete}
In the case where $G$ is discrete, one may take $U =\{e\}$. In that case $O_{\chi, U} = G$ and $L_{\chi, U} = L_\chi$. Thus, if $G$ is discrete and countable, then the assertions of Theorem~\ref{thm:closed-points} hold if and only if $G/L_\chi$ is finite, if and only if $G/\ker(\chi)$ is center-by-finite. 
\end{rem}
	
Combining Proposition~\ref{prop:tdlc-step2-nilp} with the results of the previous section, we obtain the following. 

\begin{cor}\label{cor:G-discrete}
Let $G$ be a  countable two-step nilpotent group with center $Z$. Then the following conditions are equivalent. 
\begin{enumerate}[(i)]
\item $G$ is type~$\I$. 
\item $G$ is virtually abelian. 
\item For each character $\chi \in \widehat Z$, the quotient $G/\ker(\chi)$ is center-by-finite. 
\end{enumerate}
\end{cor}
\begin{proof}
The equivalence between (i) and (ii) is valid for all discrete groups, according to Thoma's theorem \cite[Satz 4]{Tho68}. The equivalence with (iii) is a consequence of Corollary~\ref{cor:type-I-char} and Proposition~\ref{prop:tdlc-step2-nilp} (see Remark~\ref{rmk:G-discrete}). 
\end{proof}

\begin{rem}
A finitely generated nilpotent group is virtually abelian if and only if it is center-by-finite. This can be deduced for example from Bass--Guivrac'h's theorem describing the asymptotic growth type of a finitely generated nilpotent group (see \cite[Theorem~14.26]{DK18}). Thus the equivalence between (ii) and (iii) is obvious for finitely generated groups. However, a virtually abelian two-step nilpotent group that is not finitely generated may fail to be center-by-finite. As an example, let $k$ be a field of prime order $p$. Consider the action of the additive group $k$ on the $2$-dimensional vector space over $k$, denoted by $V$, where $x \in k$ acts as the matrix $\left(\begin{array}{cc} 1 & x \\ 0 & 1\end{array}\right)$. Let $A$ be the countable abelian group defined as the countably infinite direct sum of copies of $V$, and let $G = A \rtimes k$ be the semi-direct product, where $k$ acts on each copy of $V$ through the action defined above. Then $G$ is virtually abelian and two-step nilpotent, but the center of $G$ is of infinite index. 
\end{rem}
			
We mention a purely algebraic consequence of Corollary~\ref{cor:G-discrete} that we will need later. It can be compared to the fact that an infinite direct sum of non-abelian groups is never virtually abelian. 

\begin{cor}\label{cor:non-v-ab}
Let $G$ be a group and for each integer $n \geq 0$, let $G_n$ be a  non-abelian two-step nilpotent   subgroup of $G$. Suppose that $[G_m, G_n]= \{e\}$ for all $m \neq n$. Set $A_n = [G_n, G_n]$. If the subgroup generated by $\bigcup_n A_n$ is isomorphic to the direct sum $\bigoplus_n A_n$, then $G$ is not virtually abelian. 
\end{cor}
\begin{proof}
Let $g_n, h_n \in G_n$ be non-commuting elements. Let $G^1_n = \langle g_n, h_n \rangle$, $A^1_n =  [G^1_n, G^1_n]$ and $G^1 = \langle \bigcup_n G^1_n \rangle$. It suffices to show that $G^1$ is not virtually abelian. Since the $G_n$'s pairwise commute, the multiplication map defines a surjective homomorphism $\bigoplus_n G^1_n \to G^1$. Since $G^1_n$ is two-step nilpotent for all $n$, it follows that $G^1$ is a countable two-step nilpotent group. 

Set $A^1 = \langle \bigcup_n A^1_n \rangle$ and $a_n = [g_n, h_n]$. The hypotheses imply that $A^1$ is abelian and splits as the direct sum of the $A^1_n$'s. Hence its dual is the direct product $\prod_n \widehat{A^1_n}$. We infer that there is a character $\chi$ of $A^1$ such that $\chi(a_n) \neq 1$ for all $n$. We extend $\chi$ to a character defined on the center of $G^1$, that we also denote by $\chi$. We claim that $G^1/\ker(\chi)$ is not center-by-finite. 

Let $Z$ be the center of $G^1/\ker(\chi)$ and suppose for a contradiction that $Z$ has finite index. Let 
$\pi \colon G^1 \to G^1/\ker(\chi)$ be the   canonical projection. 
By the pigeonhole principle, there is a strictly increasing function $\psi \colon \NN \to \NN$ such that the map $n \mapsto \pi(g_{\psi(n)})Z$ is constant. Thus $\pi(g_{\psi(0)}^{-1} g_{\psi(n)}) \in Z$. In particular, for all $n >0$,  the commutator 
$$[\pi(g_{\psi(0)}^{-1} g_{\psi(n)}), \pi(h_{\psi(n)})] = \pi([g_{\psi(0)}^{-1}g_{\psi(n)}, h_{\psi(n)}])$$ 
is trivial. Since $G^1$ is two-step nilpotent, and since $g_{\psi(0)}$ commutes with $h_{\psi(n)}$ for $n>0$, we infer   that $\pi([g_{\psi(n)}, h_{\psi(n)}] )$ is trivial, hence that $[g_{\psi(n)}, h_{\psi(n)}]  = a_{\psi(n)} \in \ker(\chi)$ for all $n >0$. This contradicts the definition of $\chi$. 

In view of Corollary~\ref{cor:G-discrete}, we deduce that $G^1$ is not virtually abelian. The conclusion follows. 
\end{proof}
		
\begin{cor}\label{cor:G-tdlc}
Let $G$ be second countable two-step nilpotent tdlc group. Then $G$ is type~$\I$ if and only if every discrete quotient of every open subgroup of $G$ is type~$\I$. 
\end{cor}
\begin{proof}
The type~$\I$ condition passes to quotient groups and to open subgroups (see \cite[Proposition 6.E.21(1)]{BH20}). Thus the `only if' part of the corollary holds. 

The converse follows by combining Corollary~\ref{cor:type-I-char}, Proposition~\ref{prop:tdlc-step2-nilp} and Remark~\ref{rem:disquo}.
\end{proof}
		
Corollary~\ref{cor:G-tdlc} may be combined with Corollary~\ref{cor:G-discrete}. We infer that for a second countable two-step nilpotent tdlc group, the type~$\I$ condition is characterized by a purely algebraic conditions to be satisfied by discrete quotients of open subgroups.

\subsection{Groups with a bilinear commutator map}\label{sec:app-type-I}

Our next goal is to apply the previous results in order to show that several naturally occurring two-step nilpotent tdlc groups are type~$\I$. 

We start with a general discussion. Let  $G$ be a group with center $Z$. Then $G$ is two-step nilpotent if and only if the commutator group $[G, G]$ is contained in $Z$. Suppose that this the case. Given any  subgroup $N$ of $G$ with $[G, G] \leq N \leq Z$, we may view $G$ as a central extension of the abelian quotient $A = G/N$ by $N$. As such, it is described by the cohomology class of a $2$-cocycle $\omega$, which is defined as a map $\omega \colon A \times A \to N$ satisfying the following cocycle identity for all $a, b, c \in A$: 
$$\omega(a+b, c) + \omega(a, b) = \omega(a, b+c) + \omega(b, c), $$
where we have used an additive notation for the group laws of the  abelian groups $A$ and $N$. We shall focus on the case where  $G$ is locally compact and $N$ is closed; the cocycle $\omega$ is then continuous. 

Keeping in mind Theorem~\ref{thm:closed-points} and its consequences established above, we see that the type~$\I$ property of $G$ does not formally depend on the isomorphism type of $G$, or of the cohomology class of $\omega$, but rather on the commutator map 
$$G \times G \to [G, G] : (g, h) \mapsto ghg^{-1}h^{-1}.$$ 
This map is constant on cosets of $Z$, hence on cosets of $N$. Setting $A = G/N$ as above, we see that it induces canonically a map 
$$\gamma \colon A \times A \to N$$
that we also call the \textbf{commutator map}.  

\begin{rem}
It is important to observe that any biadditive map $A \times A \to N$ satisfies the cocycle identity. Since the commutator map on a two-step nilpotent group is biadditive, it follows that it can be viewed as a $2$-cocycle. Thus every two-step nilpotent group $G$ defines another two-step nilpotent group $\widetilde G$, with the same underlying set as $G$, and whose defining $2$-cocycle is the commutator map of $G$. By considering the dihedral group of order~$8$ and the quaternion group of order~$8$, it is easy to see that two non-isomorphic groups may share the same commutator map, and that $G$ need not be isomorphic to $\widetilde G$. If the map $a \mapsto 2a$ defines an automorphism of $N$, then $\widetilde G$ is isomorphic to $\widetilde{\big(\widetilde G\big)}$, see \cite[Proposition~2.8]{Con63}. We do not need those facts, but we find it relevant to record that the commutator map captures enough information to determine whether $G$ is type~$\I$, but not enough information a priori to determine the isomorphism type of $G$. 
\end{rem}

Let us start by observing that Corollary~\ref{cor:type-I-char} recovers a couple of well-known cases of type~$\I$ nilpotent groups. The fact that connected nilpotent locally compact groups are type~$\I$ is due to Dixmier~\cite{Dix59}. For $p$-adic groups, the corresponding result follows from \cite{Fel62} (see \cite[$\S$4]{Mor65} for further details). 

\begin{cor}\label{cor:connected+p-adic}
Let $G$ be two-step nilpotent locally compact group with center $Z$, and $N$ be a closed subgroup with $[G, G] \leq N \leq Z$. Then $G$ is type~$\I$ in each of the following cases. 
\begin{enumerate}[(1)]
\item $G/N$ is almost connected (i.e. the group of components of $G/N$ is compact).

\item $G/N$ has a finite index open subgroup isomorphic to $\Q_p^d \times \ZZ_p^e$ for some prime $p$ and some integers $d, e \geq 0$. 
\end{enumerate}
\end{cor}
\begin{proof}
Set $A = G/N$. Let $\chi \in \widehat Z$ be a character. In view of Corollary~\ref{cor:type-I-char}, we must show that the image of the map $\omega_\chi \colon G \to Z^\perp$ from Equation~\ref{eqn:wchi} is closed. Since $[G, G] \leq N \leq Z$ by hypothesis, we have $Z^\perp \subseteq N^\perp$. Moreover $\omega_\chi$ is trivial on $Z$-cosets, hence also on $N$-cosets. Thus we may naturally view $\omega_\chi$ as a map from $A = G/N$ to its dual $\widehat A$.

Suppose first that $A = G/N$ is almost connected. In particular it is compactly generated. By  \cite[Theorem~23.11]{Str06}, the  group $A $ splits as the direct product $M \times V \times D$, where $M$ is the largest compact subgroup of $A $,  $V \cong \R^d$ is a vector group for some $d \geq 0$, and $D \cong \mathbb Z^c$ is discrete and free abelian. Since $A$ is almost connected, it does not have any infinite cyclic quotient, so $D$ must be trivial. Hence $A \cong M \times V$. In particular $A/V$ is compact. 


It follows from Pontryagin duality that the dual $\widehat A$ splits as the direct product $   V^\perp \times \widehat V $, see \cite[Theorem~24.10]{Str06}. Since $\R$ is self-dual, so is $V$, hence $\widehat V \cong V$. Moreover $V^\perp$ is isomorphic to $\widehat{A/V}$, which is discrete since $A/V$ is compact. 

In order to show that the continuous homomorphism $\omega_\chi \colon A \to \widehat A$ has a closed image, it suffices to show that its restriction to the cocompact subgroup $V$ has a closed image. Since $V$ is connected, the restriction of $\omega_\chi$ to $V$ yields a continuous homomorphism $V \to \widehat V \cong V$. Since those groups are uniquely divisible, any group homomorphism is a morphism of $\Q$-modules, hence an $\R$-linear map by continuity (see \cite[Theorem~24.6]{Str06}). In particular its image is a vector subspace, hence it is closed. This proves the statement under the hypothesis (1). 

Suppose now that $A = G/N$ is virtually a $p$-adic group of the form $A_0 = \Q_p^d \times \ZZ_p^e$. The proof in that case follows a similar pattern. Indeed it suffices to show that the restriction of $\omega_\chi$ to the vector subgroup $\Q_p^d$ has a closed image. Since $\Q_p$ is self-dual, the dual of $A_0$ is isomorphic to $\Q_p^d \times (C_p^\infty)^e$, where $C_p^\infty$ is the dual of $\ZZ_p$, which is isomorphic to group of $p$-power roots of unity. Since $\Q_p^d$ is divisible while $C_p^\infty$ has no non-trivial divisible element, it follows that the restriction of $\omega_\chi $ yields a continuous homomorphism of $\Q_p^d$ to itself. As in the case of $\R$, any such homomorphism is a morphism of $\Q$-modules by divisibility, hence an $\Q_p$-linear map by continuity. Thus the restriction of $\omega_\chi $  to the vector subgroup $\Q_p^d$ has a closed image, and the result follows since $\Q_p^d$ is cocompact in $A$.
\end{proof}

The following result will allow us  to cover broader families of examples including algebraic groups over fields of positive characteristic.  

\begin{thm}\label{thm:bilinear}
Let $k$ be a non-discrete locally compact field. 
Let $G$ be two-step nilpotent locally compact group with center $Z$, and $N$ be a closed subgroup with $[G, G] \leq N \leq Z$. Suppose that the abelian groups $A = G/N$ and $N$  are the additive groups of vector spaces over $k$.  

If the commutator map $\gamma \colon A \times A \to N$ is $k$-bilinear, then $G$ is type~$\I$. 
\end{thm}
\begin{proof}
Since $A$ and $N$ are locally compact, they are finite-dimensional over $k$. In particular $G$ is second countable. 

Let $\chi \in \widehat N$ be a character. In view of Corollary~\ref{cor:type-I-char}, we must show that the image of the map $\omega_\chi \colon G \to N^\perp \cong \widehat{G/N}$ from Equation~\ref{eqn:wchi}  is closed (see Remark~\ref{rem:N}). 

The dual $\widehat A \cong \widehat N$  is naturally a $k$-vector space, with scalar multiplication defined by $\lambda \psi(a) = \psi(\lambda a)$ for any $\psi \in \widehat{A}$, $\lambda \in k$ and $a \in A$. The locally compact field $k$, hence also the group $A$,  are isomorphic to their dual (see Theorem~3 on p.~40 in \cite{Weil}). Now, using the $k$-bilinearity of $\gamma$, we observe that 
\begin{align*}
\lambda \omega_\chi(g)(a) & = \chi \circ \gamma(g, \lambda a) \\
& = \chi \circ \gamma(\lambda g,a)\\
& = \omega_\chi(\lambda g)(a)
\end{align*}
for all $\lambda \in k$ and $g, a \in N$. Thus the image of $\omega_\chi$ is invariant under scalar multiplication. Since  $\omega_\chi$ is a group homomorphism, its image is also stable under addition. Therefore, the image of $\omega_\chi$ is a vector subspace, hence the zero set of a family of $k$-linear forms. This implies that it is closed. 	
\end{proof}

We shall now establish Corollary~\ref{corintro:algebraic} stated in the introduction. 	
In the case~(3), the proof uses the notion of a \textbf{Moufang set}. We refer to \cite{DMS09} for an introduction to this topic. We do not repeat the detailed definitions, some of which are  rather technical, but we only describe the key points from which the desired conclusion follows. Let us moreover note that special cases of  Moufang sets of skew-hermitian type appearing below are those associated with a  separable quadratic extension of $k$, defined in \cite[Def.~5.4.1]{DMS09}. The latter suffices to treat the unipotent radical of minimal parabolics in the unitary groups $SU_3$.

\begin{proof}[Proof of Corollary~\ref{corintro:algebraic}]
We must show that the hypotheses of Theorem~\ref{thm:bilinear} are satsified. 

For the groups satisfying (1) or (2), this follows from the definitions. 

For the groups satisfying (3), the verification is more involved. Clearly we may assume that the group is non-abelian. Moreover, the case where $k = \R$ or $\C$ being covered by Corollary~\ref{cor:connected+p-adic}, we may further assume that $k$ is a non-Archimedean local field (see \cite[Chapter~1]{Weil}).

Let us first observe that such a group $U$ can indeed be viewed as a central extension of vector groups over $k$: indeed, this follows from \cite[\S 3.17]{BT65}. To check the bilinearity of the commutator map, we invoke the classification of the simple algebraic $k$-groups of $k$-rank one. Since $k$ is a non-Archimedean local field, all such groups are of classical type (see the tables in \cite[\S 4]{Tit79}). Moreover, by \cite[Chapter~1]{BT87}   (see also \cite[Remark~4.5]{TDM} and  \cite[Corollary~1.4.5]{BDMS19}), the classification implies that $U$ may be viewed as the root group of a Moufang set of skew-hermitian type, as defined in \cite[\S 4.1]{TDM}.  Such a Moufang set is determined by a pseudo-quadratic space $(D, D_0, \sigma, V, p)$ as defined in \cite[(11.17)]{TW02}. By definition, this means that  $D$ is a skew-field with center $K$ which, in our case, is finite-dimensional over $K$, $\sigma$ is an involution of $D$, $D_0$ is a $K$-subspace of $D$ containing $1$,  such that $a^\sigma D_0 a \subseteq D_0$ for all $a \in D_0$, $V$ is a right $D$-module, and $p$ is a pseudo-quadratic form with corresponding skew-hermitian form $h$. Moreover, the defining ground field $k$ coincides with the fixed field $k =  \mathrm{Fix}_K(\sigma)$ of $\sigma$, which is of codimension~$1$ or $2$ in $K$. 

The group $U$ is defined as in   \cite[(11.24)]{TW02}. It is the set 
$$U = \{(v, a) \in V \times D \mid p(v) - a \in D_0\}$$ 
with group multiplication 
$$(v, a).(w, b) = (v+w, a+b+h(w, v)).$$
We define the closed normal subgroup $N$ as $\{(0, a) \in U  \mid a \in D_0\}$. Observe that $N \cong D_0$ is a $K$-vector group, hence a $k$-vector group. One checks that $U/N$ is also a $k$-vector group with scalar multiplication induced by $\lambda.(v, a) = (v\lambda, a\lambda^2)$. Moreover, one computes that the commutator map is 
\begin{align*}
\gamma \colon \qquad  U/N \times U/N & \to D_0 \\
 \big((v, a)N , (w, b)N\big) &\mapsto h(w, v) - h(v, w) = h(w, v) + h(w, v)^\sigma.
\end{align*}
This need not be bilinear over $K$, but it is bilinear over  $k$.
\end{proof}


\section{Contraction groups}	\label{sec:contraction}

\subsection{Definition}	
In this section, we address the type~$\I$ dichotomy for some tdlc contraction groups.

Let us start by recalling the definition. It is assumed throughout that  automorphisms of topological groups are bicontinuous.

\begin{dfn}
Let $G$ be a topological group and $\alpha \in \Aut(G)$. The automorphism $\alpha$ is \textit{contractive} if for all $g \in G$, $\alpha^n(g) \rightarrow \id_G$ as $n \rightarrow \infty$. A \textit{contraction group} is a pair $(G,\alpha)$, where $G$ is a topological group, and $\alpha \in \Aut(G)$ is contractive. If $G$ satisfies a property P (e.g.\ locally compact, totally disconnected, torsion...) then we call $(G,\alpha)$ a P contraction group.
\end{dfn}

The structure theory of locally compact contraction groups is studied thoroughly in the articles \cite{Sie86,GW10,GW21,GW21b}.	In view of the important role played by contraction groups within the structure theory of general tdlc groups, studying their representation theory is rather natural (see  \cite{Car24} for more information). Contraction groups also appear naturally in algebraic groups over local fields; indeed, unipotent radicals of parabolic subgroups are all contraction groups (see \cite[Lemma~2.4]{Pra77}). In particular, the examples treated by Corollary~\ref{corintro:algebraic}(3) are all contraction groups.

\subsection{Two-step nilpotent groups with monomial commutation relations}

We shall consider the family of two-step nilpotent  groups defined as follows. 

\begin{dfn}\label{def:monomial}
Let $p$ be a prime, $\F_p$ be the finite field of order $p$, and $A = \F_p(\!(t)\!)$ the local field of formal Laurent series with coefficients in $\F_p$. Let $G$ be a two-step nilpotent with center $Z$ and $N$ be a closed subgroup with $[G, G] \leq N \leq Z$, such that $G/N$ and $N$ are both isomorphic to the additive group $A$. Denote by $\gamma \colon A \times A \to A$ the commutator map. We say that $G$ has \textbf{monomial commutation relations} if there exists a map $u \colon \ZZ \times \ZZ \to \F_p$ such that
$$\gamma(t^m, t^n ) = u(m, n) t^{m+n}$$
for all $m, n \in \ZZ$. Notice that such a group $G$ is second countable, since $A$ is so. 
\end{dfn}

Since the commutator map is continuous and bi-additive, the identity $\gamma(t^m, t^n ) = u(m, n) t^{m+n}$ completely determines the map $\gamma$. 
	
In order for  a group $G$ as in Definition~\ref{def:monomial} to admit a contractive automorphism that preserves $N$, it is necessary that $A$ has two contractive automorphisms, say $\alpha, \beta$, such that $\gamma(\alpha(g), \alpha(h)) = \beta(\gamma(g, h))$ for all $g, h \in G$. The most natural contractive automorphisms are given by the multiplication by  positive powers of $t$. If $\alpha \colon a \mapsto t^d a$ and $\beta \colon a \mapsto t^{2d} a$, we see that the map $u$ from Definition~\ref{def:monomial} must satisfy
$$u(m+d, n+d) = u(m, n)$$
for all $m, n \in \ZZ$. This is clearly the case if $u(m, n) = \sigma_{m-n}$, where $\sigma \colon \ZZ \to \F_p : z \mapsto \sigma_z$ is a bi-infinite sequence of elements of $\F_p$. Since the map $u$ satisfies $u(n, n)=0$ and $u(m, n) = -u(n, m)$ for all $m, n$, we see that $\sigma_0 = 0$ and $\sigma_{-z} = -\sigma_z$ for all $z$. Thus $\sigma$ is uniquely determined by its restriction to the positive integers. If $G$ is as in Definition~\ref{def:monomial} and if moreover $u(m, n) = \sigma_{m-n}$, we say that $G$ has \textbf{monomial commutation relations of type~$\sigma$}.

\begin{thm}\label{thm:monomial-type-I}
Let $G$ be a two-step nilpotent tdlc group with monomial commutation relations of type~$\sigma$. Then $G$ is type~$\I$ if any of the following conditions are satisfied for some integer $c>0$. 
\begin{enumerate}[(1)]
\item  $\sigma_z \neq 0$ for all $z \in \ZZ$ with $|z| > c$. 

\item $\sigma_z = 0$ for all  $z \in 2\ZZ$, and $\sigma_z \neq 0$ for all  $z \in 2\ZZ+1$ with $|z| > c$. 

\item There is an integer $d>0$ such that $\sigma_z = 0$ for all $z \not \in d\ZZ$, and  $\sigma_z \neq 0$ for all   $z \in d\ZZ$ with $|z| > c$. 
\end{enumerate}
\end{thm}

\begin{proof}
Since $A \cong N$ is a group of exponent $p$, the image of every non-trivial character is a cyclic group of order~$p$. We may thus view the Pontryagin dual $\widehat N$  of $N$ as the group of  continuous homomorphisms $\chi \colon N \to \F_p$. 

In order to show that $G$ is type~$\I$, we shall apply Proposition~\ref{prop:tdlc-step2-nilp} for an arbitrary character $\chi \in \widehat N$ (see Remark~\ref{rem:N}). We may assume without loss of generality that $\chi$ is non-trivial. Set $a_m = \chi(t^m)$ for all $m \in \ZZ$. Since $\chi$ is continuous, we have $a_m = 0$ for all sufficiently large $m$. Thus it makes sense to define $k_0 \in \ZZ$ as the largest integer such that $a_{k_0} \neq 0$. We also set $K_0 = \max\{-1, k_0\}$ and set 
$$U^A = t^{K_0+1}\F_p[\![t]\!].$$
Hence $U^A$ is a compact open subgroup of $A$, and there exists a compact open subgroup $U$ of $G$ whose projection to $A \cong G/N$ coincides with $U^A$. Moreover, it follows from the definitions that $[U, U] \subseteq \ker(\chi)$. 

Let $O_{\chi, U} = \{g \in G\mid [g, U] \subseteq \ker(\chi)\}$, as in Proposition~\ref{prop:tdlc-step2-nilp}. Let also $O_{\chi, U}^A \leq A$ be the image of   $O_{\chi, U}$ under the canonical projection $G \to G/N \cong A$. An element $x = \sum_{i=i_0}^\infty x_i t^i  \in A$ belongs to $O_{\chi, U}^A $ if and only if $\gamma(x, t^m) \in \ker(\chi)$ for all $m > K_0$. Using that $\chi$ is continuous and $\F_p$-linear, we deduce that 
\begin{align}\label{eq:1}
x \in O_{\chi, U}^A  \Longleftrightarrow 
\sum_{i=i_0}^\infty \sigma_{i-m} a_{i+m} x_i = 0 \hspace{1cm} \forall m > K_0. 
\end{align}
It is important to keep in mind  that $a_k =0$ for all $k > k_0$. Thus, in the equations above, the terms indexed by any $i > k_0 - m$ are all zero, so that the series is actually a finite sum. Thus we obtain 
\begin{align}\label{eq:2}
x \in O_{\chi, U}^A  \Longleftrightarrow 
\sum_{i=i_0}^{k_0-m} \sigma_{i-m} a_{i+m} x_i = 0 \hspace{1cm} \forall m > K_0.
\end{align}
We view that condition as a system of $\F_p$-linear equations that the coefficients $(x_i)_{i\geq i_0}$ of $x$ must satisfy. 
Using again that $a_k =0$ for all $k > k_0$, we see that all the equations corresponding to any $m > k_0 - i_0$ are trivial. Thus the system has only finitely many equations, one for each value of $m$ satisfying $K_0 < m \leq k_0 - i_0$. We rewrite the system by numbering the equations by the number $\ell = k_0 - i_0 -m$ and by indexing the sum over  $n= i - i_0$. We obtain that the condition $x \in O_{\chi, U}^A$ holds if and only if 
\begin{align}\label{eq:3--}
\sum_{n=0}^{\ell} \sigma_{2i_0 + \ell - k_0 + n}  a_{k_0 - \ell + n } x_{i_0 + n} = 0 \hspace{1cm} \forall \ell \in \{0, 1, \dots,   k_0  -i_0 - K_0 -1\}.
\end{align}
We may write this system of linear equations in matrix form as follows:
\begin{align}\label{eq:3}
\left( 
\begin{matrix}
\sigma_{2i_0-k_0} a_{k_0} & 0 & 0 & \cdots \\
\sigma_{2i_0-k_0+1} a_{k_0-1} & \sigma_{2i_0-k_0+2} a_{k_0} & 0 & \cdots \\
\sigma_{2i_0-k_0+2} a_{k_0-2} & \sigma_{2i_0-k_0+3} a_{k_0-1} &  \sigma_{2i_0-k_0+4} a_{k_0} & \cdots \\
\vdots & \vdots & \vdots & \ddots
\end{matrix}
\right) 
\left( 
\begin{matrix}
x_{i_0}\\
x_{i_0+1}\\
x_{i_0+2}\\
\vdots
\end{matrix}
\right)
= 
\left(
\begin{matrix}
0\\
0\\
0\\
\vdots
\end{matrix}	
\right)
\end{align}

Observe that the matrix in that system (\ref{eq:3}) is lower triangular. For $\ell = 0, 1, \dots$, the  diagonal entry on the $(\ell+1)^{\text{st}}$ row is 
$$\sigma_{2i_0  - k_0 + 2\ell}  a_{k_0 }.$$
Moreover we have $a_{k_0} \neq 0$. 

\medskip
Assume now that the condition (1) holds. Then the number of rows with a zero diagonal coefficient is at most $c+1$. Thus the space of solutions of the system (\ref{eq:3}) is finite-dimensional, and its dimension  is bounded independently of $i_0$.  This implies that the $\F_p$-dimension of $O_{\chi, U}^A$ is finite. Therefore $O_{\chi, U}/N$ is finite. Since  the group $L_{\chi, U}$ defined in Proposition~\ref{prop:tdlc-step2-nilp} contains $N$, we infer that $O_{\chi, U}/L_{\chi, U}$ is finite. In view of Proposition~\ref{prop:tdlc-step2-nilp}  and Corollary~\ref{cor:type-I-char}, this confirms that $G$ is type~$\I$. 

\medskip
Assume now that the condition (2) holds. We distinguish two cases. 

Suppose first that $a_k = 0$ for all odd $k$. Observe that the integers $i-m$ and $i+m$ have the same parity. By assumption $\sigma_k = 0$ for all even $k$ and $a_k=0$ for all odd $k$. It follows that all the linear equations in the system (\ref{eq:1}) above are trivial in the case at hand. Hence $O_{\chi, U}^A = A$ so that $O_{\chi, U} = G$. Now an element $x \in A$ belongs to the projection of $L_\chi$ if and only if it satisfies the system  (\ref{eq:1}) for all $m \in \ZZ$. Every equation in that system being trivial, we infer that $L_\chi = G$, hence $L_{\chi, U} = G$, so that the quotient $O_{\chi, U}/L_{\chi, U}$ is trivial in this case. 

Suppose now that there exists an odd integer $k$ such that $a_k \neq 0$. We let $\ell_0$ be the largest such integer. In particular, we have $\ell_0 \leq k_0$. It follows that the first $k_0 - \ell_0$ rows of the matrix of the system~(\ref{eq:3}) are trivial. After discarding them, we obtain a lower triangular matrix whose diagonal entries are of the form 
$$\sigma_{2i_0  +2\ell - 2k_0 + \ell_0}  a_{\ell_0 }.$$
As in case (1) above, we deduce from the hypotheses that the number of rows with a zero diagonal coefficient is  bounded independently of $i_0$. As above, this implies that the $\F_p$-dimension of $O_{\chi, U}^A$ is finite, so that $O_{\chi, U}/L_{\chi, U}$ is finite also in this case.

In particular $O_{\chi, U}/L_{\chi, U}$ is finite in both cases, hence $G$ is type~$\I$ by Proposition~\ref{prop:tdlc-step2-nilp}  and Corollary~\ref{cor:type-I-char}. 

\medskip 
Assume finally that the condition (3) holds. For each $i \in \{0, 1, \dots, d-1\}$, we set $A_i = t^i \F_p(\!(t^d)\!) \leq A$ and define $G_i$ as the preimage of $A_i$ under the canonical projection $G \to G/N \cong A$. Each $G_i$ is a closed normal subgroup of $G$, and we have $G = G_0 G_1 \dots G_{d-1}$. Moreover, in view of (3), we have $[G_i, G_j]= \{e\}$ for all $i \neq j$. It follows that the product map
$$G_0 \times G_1 \times \dots \times  G_{d-1} \to G : (g_0, g_1, \dots, g_{d-1}) \mapsto g_0 g_1\dots g_{d-1}$$ 
is a continuous surjective homomorphism. Recall that the type~$\I$ property is preserved under forming direct products (see \cite[Proposition~6.E.21(3)]{BH20}). Therefore, to show that $G$ is type~$\I$, it suffices to show that $G_i$ is type~$\I$ for each $i$.

Fix $i \in \{0, 1, \dots, d-1\}$. For all $a, b \in \ZZ$, we have $t^{i+ad}, t^{i+bd} \in A_i$ and 
$$\gamma(t^{i+ad}, t^{i+bd}) = \sigma_{(a-b)d} t^{2i + (a+b)d}.
$$ 
Set $N_2 = t^{2i}\F_p(\!(t^d)\!)$ and $N_1 = \bigoplus_{j \not \equiv 2i \mod d} t^j \F_p(\!(t^d)\!)$. Thus $\F_p(\!(t)\!)\cong  N \cong N_1 \times N_2$, and we have $[G_i, G_i] \leq N_2$. In view of Corollary~\ref{cor:type-I-quot}, we infer that $G_i$ is type~$\I$ if and only if $G_i/N_1$ is so. 

We claim that $G_i/N_1$ has monomial commutation relations. In order to check this, we view $N/N_1$ as a subgroup of $G_i/N_1$. We identify the quotient $A_i \cong G_i/N \cong (G_i/N_1)/(N/N_1)$ with $A$ via the isomorphism sending $t^{i+nd}$ to $t^n$ for all $n \in \ZZ$. We moreover identify the group $ t^{2i}\F_p(\!(t^d)\!) = N_2 \cong N/N_1$ with $A$ via the isomorphism sending $t^{2i+nd}$ to $t^n$ for all $n \in \ZZ$. After those identifications, we see that the commutator map of $G_i$ satisfies the identity
$$\gamma(t^a, t^b) = \sigma_{(a-b)d} t^{a+b}
$$ 
for all $a, b \in \ZZ$. Now the condition (3) implies that $G_i$ satisfies the condition (1). Thus it is type~$\I$, and the proof is complete. 
\end{proof}

\begin{rem}
Observe that in all cases covered by Theorem~\ref{thm:monomial-type-I}, the restriction of $\sigma$ to $\NN$ is an asymptotically periodic sequence, which is not asymptotically zero. That condition is however not sufficient for $G$ to be type~$\I$, by virtue of the following result.
\end{rem}
%
%
%
%

\begin{thm}\label{thm:non-type-I} 
Let $G$ be a two-step nilpotent tdlc group with monomial commutation relations of type~$\sigma$. Let $d>0$ be an integer with $\sigma_d \neq 0$. If $\sigma_{dn} = 0$ for all $n \geq 2 $, then $G$ is not type~$\I$. 
\end{thm}

\begin{proof}
We shall prove that $G$ has an open subgroup $O$ admitting a discrete quotient that is not type~$\I$. The required conclusion follows (see Corollary~\ref{cor:G-tdlc}). 

Let $U^A = \F_p[\![t ] \!]$ and $U$ be a compact open subgroup of $G$ whose image under the canonical projection $G \to G/N\cong A$ is $U^A$. Upon replacing $U$ by a possibly larger group, we may assume that $U \cap N = t^{-n_0}\F_p[\![t ] \!]$ for some $n_0 >0$. 

We also define
$$O^A = \bigg(\bigoplus_{n > 0} \langle t^{-3nd} \rangle \oplus \langle  t^{-3nd - d} \rangle \bigg) \oplus U^A,$$
and $O$ be the preimage of $O^A$ in $G$. Finally, we define 
$$N_1 = \bigg(\bigoplus_{n > 0, n \not \equiv -d \mod 6d} \langle t^{-n} \rangle \bigg) \oplus \F_p[\![t ] \!],$$
viewed as a subgroup of $N \cong A$.  Thus $N_1$ is contained in the center of $G$. 

We claim that $[O, U] \subseteq N_1$. Since the commutator is constant on cosets of $N$ in $G$, it suffices to show that for all integers $a, b \in \ZZ$ with $t^a \in O^A$ and $t^b \in U^A$, we have $\gamma(t^a, t^b) \in N_1$. 

Let $t^a \in O^A$ and $t^b \in U^A$. 
We have $b \geq 0$. If $a \geq 0$, then the fact that $\gamma(t^a, t^b) \in N_1$ is clear. If $a < 0$, we have $a = -3nd$ or $a = -3nd -d$ for some $n > 0$. We infer that $\gamma(t^a, t^b ) = - \sigma_{b+3nd} t^{b-3nd}$ or $-\sigma_{b+3nd+d} t^{b-3nd-d}$. If $b $ is a multiple of $d$, then $\sigma_{b+3nd} = \sigma_{b+3nd+d} = 0$ by hypothesis, hence the required assertion follows. If $b $ is not a multiple of $d$,  then neither $b - 3nd$ nor $b -3nd -d$ is congruent to $-d$ modulo $6d$. Thus we obtain that $\gamma(t^a, t^b) \in N_1$ in all cases. 

The group $N_1$ is a central subgroup of $O$ and $[O, U] \subseteq N_1$, it follows that $UN_1$ is an open normal subgroup of $O$. Let now 
$$N_2 = \bigoplus_{n > n_0, n  \equiv -d \mod 6d} \langle t^{-n} \rangle.$$
Hence $N_2 \cap UN_1 = \{e\}$. Finally, for each $n >n_0$, let $g_n$ and $h_n \in O$ be preimages of $t^{-3nd}$ and $t^{-3nd -d} \in O^A$ respectively. We have $[g_n, h_n] = \gamma( t^{-3nd}, t^{-3nd -d}) = \sigma_d t^{-d - 6nd} \in N_2$. By hypothesis $\sigma_d \neq 0$, hence $g_n$ and $h_n$ do not commute. On the other hand, $g_n$ commutes with $g_m$ and $h_m$ for all $m \neq n$. 

Set $Q = O/UN_1$ and let $\pi \colon O \to Q$ be the canonical projection. For $n> n_0$, let $Q_n$ be the subgroup of $Q$ generated by $\pi(g_n)$ and $\pi(h_n)$. Since the restriction of $\pi$ to $N_2$ is injective, it follows that the groups $Q_n$, which commute pairwise, are non-abelian. Moreover, the collection of their commutator groups $[Q_n, Q_n]$ taken over all $n > n_0$ generates their direct sum.  In view of Corollary~\ref{cor:non-v-ab}, we infer that $Q$ is not virtually abelian, hence it is not type~$\I$ (see Corollary~\ref{cor:G-discrete}). 
%
%
\end{proof}

The following consequence of Theorem~\ref{thm:non-type-I} is immediate. This recovers the result \cite[Theorem 5.9]{Car24} and extends it to the framework of two-step nilpotent groups with monomial commutation relations. 

\begin{cor}
Let $G$ be a two-step nilpotent tdlc group with monomial commutation relations of type~$\sigma$. If $\sigma$ is non-zero and finitely supported, then $G$ is not type~$\I$.
\end{cor}


\subsection{The Gl\"ockner--Willis contraction groups}\label{sec:GW}

We now apply our results to a class of two-step nilpotent groups first defined by Gl\"ockner--Willis in  \cite[Section~8]{GW21}. The study of the unitary representations  of those groups was initiated by the second author in \cite{Car24}. Our goal is to strengthen and extend the results from loc.\ cit.

As before, we let $p$ be a prime and set $A= \F_p(\!(t)\!)$. Let $\nu \colon A \to \ZZ$ be the standard valuation. 


\begin{dfn}
Given a sequence $s \colon \NN_{>0} \to  \{0,1\}$, we define the map 
$$\eta_s \colon A \times A \rightarrow A$$ 
on elements $x := \sum_{i = \nu(x)}^\infty x_it^i, y := \sum_{i = \nu(y)}^\infty y_it^i \in \mathbb{F}_p(\!(t)\!)$ by
\begin{displaymath} 
\eta_s(x,y) := \sum_{k \in \supp(s)} \sum_{i = \nu(x)}^\infty x_i y_{i+2k} t^{i+k}.
\end{displaymath}
\end{dfn}

The following is shown in \cite{GW21}.

\begin{prop}[{See \cite[Section~8]{GW21}}] \label{prop:GW}
For any $s \in \{0,1\}^{\mathbb{N}_{>0}}$, the map $\eta_s$ is a continuous bi-additive $2$-cocycle on $A$, and it is equivariant in the sense that $\eta_s(tx,ty) = t\eta_s(x,y)$ for all $x,y \in A$. Furthermore, the corresponding central extension, denoted $G = A \times_{\eta_s} A$, is a contraction group with respect to the automorphism $\alpha$ that is multiplication by $t$ on each  factor. For distinct sequences  $s \in \{0,1\}^{\mathbb{N}_{>0}}$, the corresponding groups are not isomorphic.
\end{prop}

Since $\eta_s$ is continuous and bi-additive, it is determined by its values on pairs $(t^a, t^b)$. We observe that 
$$
\eta_s(t^a, t^b) = \left\{
\begin{array}{cl}
s(k) t^{a+k} & \text{if } b = a+2k \text{ for some } k >0,\\
0 & \text{otherwise} 
\end{array}
\right.
$$
for all $a, b \in \ZZ$.

The group $G = A \times_{\eta_s} A$ is defined as the set $A \times A$ endowed with the multiplication defined by
$$
(x, a).(y, b) = (x+y, a+b + \eta_s(x, y)).
$$
In particular,  the second factor $N = \{(0, a) \mid a \in A\}$ in $G$ corresponds to the canonical central subgroup  isomorphic to $A$, and $G/N \cong A$. Viewing those isomorphisms as identifications, we may consider the commutator map $\gamma_s \colon A \times A \to A$ associated with $G$ as in Section~\ref{sec:app-type-I}. 
One can check that the following identities hold for all $a, b \in \mathbb{Z}$:
\begin{align} \label{eq:comm-GW}
\gamma_s(t^a,t^b)= \left\{
\begin{array}{cl}
0 & \text{if } a = b \text{ or } a \not \equiv b \mod 2,\\
s(\frac{b-a} 2) t^{\frac{a+b} 2}  & \text{if } a < b \text{ and } a  \equiv b \mod 2,\\
 - s(\frac{a-b} 2) t^{\frac{a+b} 2} & \text{if } a > b \text{ and } a  \equiv b \mod 2.
\end{array}
\right.
\end{align}

For $i = 0, 1$, set $A_i = t^i \F_p(\!(t^2)\!)$. We have $A \cong A_0 \oplus A_1$. Let $G_i$ be the preimage of $A_i$ under the canonical projection $G \to G/N \cong A$. 

\begin{prop}\label{prop:basic-GW}
For any sequence $s\in \{0,1\}^{\mathbb{N}_{>0}}$, the group $G = A \times_{\eta_s} A$ has the following properties. 
\begin{enumerate}[(i)]
\item $G_0$ and $G_1$ are closed normal subgroups of $G$ that commute. Moreover,  we have $G = G_0 G_1$. 
\item $G_0$ is isomorphic to $G_1$. 
\item $G$ is type~$\I$ if and only if $G_0$ is type~$\I$. 
\item $G_0$ is isomorphic to a two-step nilpotent tdlc group with monomial commutation relations of type $\sigma$, where $\sigma \colon \ZZ \to \F_p$ is the sequence defined by
$$
\sigma_z = \left\{
\begin{array}{cl}
s(z) & \text{if } z > 0,\\
0 & \text{if } z= 0,\\
-s(-z) &\text{if } z< 0.
\end{array}
\right.
$$
\end{enumerate}
	 
\end{prop}
\begin{proof}
The commutation relations (\ref{eq:comm-GW}) imply that $G_0$ and $G_1$ commute.  Hence the assertion (i) follows readily from the definitions.

By Proposition~\ref{prop:GW}, the group $G$ has a contractive automorphism $\alpha$ that acts as the multiplication by $t$ on $N \cong A$ and $G/N \cong A$. It follows that the restriction of $\alpha$ to $G_0$ is an isomorphism of $G_0$ onto $G_1$. Thus (ii) holds. 

Suppose that $G_0$ is type~$\I$, so that $G_1$ is also type~$\I$ by (ii). Hence the direct product $G_0 \times G_1$ is type~$\I$ by \cite[Proposition~6.E.21(3)]{BH20}. By (i), the multiplication map yields a surjective continuous homomorphism $G_0 \times G_1 \to G$. Hence $G$ is type~$\I$ as well.

Assume conversely that $G_0$ is not type~$\I$. In particular it is non-abelian, hence $s$ is non-zero. It follows from  \cite[Lemma~8.4]{GW21} that $N$ is the center of $G$ and of $G_0$. By Corollary~\ref{cor:type-I-char}, there exists a character $\chi\in \widehat N$ such that $\omega_\chi \colon G_0/N \to \widehat{G_0/N}$ has a non-closed image. Let $U\leq G$ be a compact open subgroup of $G$ such that $[U, U]\leq \ker(\chi)$, and set $U_0 = U \cap G_0$ and $U_1 = U \cap G_1$. Since $G$ is second countable, the surjective homomorphism $G_0 \times G_1 \to G$ is an open map, hence $U_0 U_1$ is open in $G$. Therefore, upon replacing $U$ by the smaller subgroup $U_0 U_1$, we may assume that $U = U_0 U_1$. 

Let $O_{\chi, U_0} \leq G_0$ be the group defined in Proposition~\ref{prop:tdlc-step2-nilp}. The latter implies that  the quotient $O_{\chi, U_0}/ \ker(\chi)$ is not center-by-finite. Since $U = U_0 U_1$ and since $G_0$ commutes with $G_1$, we infer that $O_{\chi, U_0} \leq O_{\chi, U}$. Hence  $O_{\chi, U}/ \ker(\chi)$ is not center-by-finite. Using Proposition~\ref{prop:tdlc-step2-nilp} again, we deduce that $\omega_\chi \colon G/N \to \widehat{G/N}$ has a non-closed image. Thus $G$ is not type~$\I$ by Corollary~\ref{cor:type-I-char}. This proves that (iii) holds.

To prove (iv), we observe that $G_0$ contains $N$, and that $G_0/N \cong A_0 = \F_p(\!(t^2)\!)$. The commutator map of $G_0$ is the map $\gamma_s \colon A_0 \times A_0 \to N$ induced by the commutator relations appearing in (\ref{eq:comm-GW}). Consider the isomorphism $\delta \colon A\to A_0$ sending $t^{n}$ to $t^{2n}$ for all $n \in \ZZ$, and the isomorphism $\mu \colon A \to N$ sending $x$ to $-x$ for all $x$. Define $\gamma = \mu \circ \gamma_s \circ (\delta \times \delta) \colon A \times A \to A$. We deduce from (\ref{eq:comm-GW}) that 
$$\gamma(t^a, t^b)  = \sigma_{a-b} t^{a+b}$$
for all $a,b \in \ZZ$, where $\sigma$ is the bi-infinite sequence defined in the statement of the proposition. The assertion (iv) readily follows. 
\end{proof}
	
Theorem~\ref{thmintro:GW-type-I} stated in the introduction is now a direct consequence of Theorems~\ref{thm:monomial-type-I} and~\ref{thm:non-type-I}. 
%
%
%
%
%
%
%
%
%

\section{Type~$\I$ extensions}\label{sec:ext}

The goal of this section is to complete the proof of Theorem~\ref{thmintro:type-I-ext} stated in the introduction. 

%

\begin{proof}[Proof of Theorem~\ref{thmintro:type-I-ext}]
Set $A = G/N$. Since $A$ is abelian, the desired conclusion is clear if $\chi$ is the trivial character. We assume henceforth that $\chi$ is non-trivial. 

Let $Q = G/\ker(\chi)$ and denote the canonical projection by $g \mapsto \bar g$. Since $\ker(\chi) \leq N$, the quotient map $G \to G/N = A$ factors through $Q$. In particular, every character $\varphi \in \widehat A$ may be viewed as a character of $Q$ (and also of $G$). 

Since $N$ is of exponent $p$, it follows that $C = N/\ker(\chi) \cong \mathrm{im}(\chi)$ is a cyclic group of order~$p$.  Similarly, the image of each $\varphi \in \widehat A$ is contained in the cyclic subgroup of order~$p$ of $\C^*$, which coincides with $ \mathrm{im}(\chi)$. Upon identifying $C$ with $ \mathrm{im}(\chi)$ via $\chi$, we infer that the dual $\widehat A$ acts on the quotient group $Q= G/\ker(\chi)$ via
$$\widehat A \times Q \to Q : (\varphi, \bar g) \mapsto \varphi(\bar g) \bar g.$$
One checks that this is a continuous action by automorphisms on $Q$. Therefore, we may form the semi-direct product 
$$E = \widehat A \ltimes Q.$$
By construction, the group $Q$ embeds as a closed normal subgroup of $E$. The group $E$ consists of ordered pairs $(\varphi, \bar g) \in \widehat A \times Q$ with multiplication defined by
$$(\varphi, \bar g) \cdot (\psi, \bar h) = (\varphi \psi, \psi(\bar g)^{-1} \bar g \bar h).$$
Let $\bar C = \{(1, \bar c) \mid \bar c \in C\}$. 

\medskip
We claim that $E$ is a two-step nilpotent locally compact group. More precisely, we claim that
$[E, E] \leq \bar C \leq Z,$ where $Z$ denotes the center of $E$. 
Indeed, one computes that for all $(\varphi, \bar g) ,  (\psi, \bar h)\in E$, we have 
$$[(\varphi, \bar g) ,  (\psi, \bar h)] = \big(1, [\bar g , \bar h ] \psi(\bar g)^{-1} \varphi(\bar h)\big) \in \bar C.$$
Thus we have $[E, E] \leq \bar C$. Let us now assume that $\bar g \in C \leq Q$. It then follows that $[(1, \bar g) ,  (\psi, \bar h)] = \big(1, [\bar g , \bar h ] \psi(\bar g)^{-1} \big)= e$, since the elements of $C$ are central in $Q$, and since every $\psi \in \widehat A$, viewed as a character of $Q$, is trivial on $C$. The claim is proved. 

\medskip
We next claim that $E$ is type~$\I$. To this end, we rely on Corollary~\ref{cor:type-I-char}. Since $[E, E] \leq \bar C \leq Z$ by the previous claim, and since the map $\bar c \mapsto (1, \bar c)$ is an isomorphism of $C$ onto $\bar C$, it suffices to show that for every $\sigma \in \widehat C$, the homomorphism 
$$\omega_\sigma \colon E \to \widehat{E/\bar C} : (\varphi, \bar g) \mapsto  \omega_\sigma 
(\varphi, \bar g) \colon  \begin{array}{rcl}
E/\bar C & \to &  \C^*\\  
(\psi, \bar h)\bar C & \mapsto & \sigma\big(  [\bar g , \bar h ] \psi(\bar g)^{-1} \varphi(\bar h)\big)
\end{array}$$
has a closed image. This is clear if $\sigma$ is trivial. We finish the proof of  the claim by showing that if $\sigma$ is non-trivial, the homomorphism $\omega_\sigma$ is surjective. 

We have $E/\bar C \cong  \widehat A \times A$, so that  $\widehat{E/\bar C} \cong A \times \widehat A$. Moreover $A$ is of exponent $p$, so its dual $\widehat A$ is also of exponent~$p$. Taking $\bar g = 1$, we see that $\omega_\sigma 
(\varphi, 1)  \colon (\psi, \bar h)\bar C  \mapsto  \sigma \circ  \varphi(\bar h)$. Since $\sigma$ is injective, we have $\{\sigma \circ \varphi \mid \varphi \in \widehat A\} = \widehat A$, hence we obtain $\{\omega_\sigma 
(\varphi, 1) \mid \varphi \in \widehat A\} = \{1\} \times \widehat A$. Given $\bar g \in Q$, we now set $\varphi_{\bar g} \colon  \bar h \mapsto [\bar g, \bar h]^{-1}$. Since $C$ is identified with $\mathrm{im}( \chi)$, we may view $\varphi_{\bar g}$ as a character of $A$. It follows that for each $\bar g \in Q$, we have $\omega_\sigma 
(\varphi_{\bar g}, \bar g)  \colon (\psi, \bar h)\bar C  \mapsto  \sigma \circ  \psi(\bar g)^{-1}$. Using again that  $\sigma$ is injective, we obtain $\{\omega_\sigma 
(\varphi_{\bar g}, \bar g) \mid \bar g \in Q\} =\widehat{(\widehat  A)} \times \{1\} \cong A \times \{1\}$. It follows that the image of $\omega_\sigma$ contains the full direct product $\widehat{E/\bar C} \cong A \times \widehat A$, thereby confirming that $\omega_\sigma$ is surjective. 

\medskip We have shown that $Q$ continuously embeds as a closed normal subgroup of $E$, which is a two-step nilpotent second countable locally compact group of type~$\I$. We henceforth view $Q$ as a closed normal subgroup of $E$. Since $A$ and $N$ have exponent $p$, they are both totally disconnected, hence $G, Q$ and $E$ are all totally disconnected as well. Let $U \leq E$ be a compact open subgroup. Then $QU$ is an open subgroup of $E$, hence it is type~$\I$ by \cite[Proposition~6.E.21(1)]{BH20}. Thus $Q$ is a closed cocompact normal subgroup of $QU$, which is a  two-step nilpotent type~$\I$ group.  
\end{proof}
	
\begin{rem}\label{rem:ex-ext}
Every non-type~$\I$ contraction group $G$ afforded by Theorem~\ref{thmintro:GW-type-I} has a factor representation that is not type~$\I$. Letting $\chi$ be the central character of such a factor representation, we infer that the quotient $G/\ker(\chi)$ is not type~$\I$. Theorem~\ref{thmintro:type-I-ext} applies to $G$, hence we obtain numerous non-type~$\I$ groups admitting an extension by  a compact group that is type~$\I$.
\end{rem}


\bibliographystyle{amsalpha}
\bibliography{RepNil}


\end{document}